\documentclass[10pt]{article}
\usepackage{geometry} 
\usepackage{amssymb,amsmath}
\usepackage{graphicx}
\usepackage{algpseudocode}
\usepackage{amstext}
\usepackage{amsfonts}
\usepackage{algorithm}
\usepackage[bf,SL,BF]{subfigure}
\usepackage{graphicx}
\usepackage{amsthm}
\usepackage[nomargin,inline,marginclue,status=draft]{fixme}
\theoremstyle{plain}
\usepackage{xcolor}
\usepackage[colorlinks = true,
            linkcolor = red,
            urlcolor  = blue,
            citecolor = blue,
            anchorcolor = blue]{hyperref}
\newtheorem{definition}{Definition}[section]

\newtheorem{corollary}[definition]{Corollary}
\newtheorem{prop}[definition]{Proposition}

\theoremstyle{definition}
\newtheorem{remark}[definition]{Remark}

\title{A Note on Random Sampling for Matrix Multiplication}
\author{Yue Wu \thanks{School of Engineering, University of Edinburgh, Edinburgh, UK, yue.wu@ed.ac.uk} }

\begin{document}

\maketitle

\begin{abstract}
This paper extends the framework of randomised matrix multiplication in \cite{DKM06} to a coarser partition and proposes an algorithm as a complement to B{\scriptsize ASIC}M{\scriptsize ATRIX}M{\scriptsize ULTIPLICATION} in \cite{DKM06},  especially when the optimal probability distribution of the latter algorithm is closed to uniform. The new algorithm increases the likelihood of getting a small approximation error in 2-norm and has the squared approximation error in Frobenious norm bounded by that from algorithm B{\footnotesize ASIC}M{\footnotesize ATRIX}M{\footnotesize ULTIPLICATION}.

\vspace{0.0cm}

\end{abstract}

\section{Introduction}\label{sec:intro}
In order to ease the heavy computational burden from massive matrix problems, randomised algorithms for matrix multiplication of large scale has been developed well in literature. For instance, the framework provided in \cite{DK01} and \cite{DKM06} shed a light on approximating matrix product $AB$ in spirit of Monte Carlo method, where $A$ is an $m\times n$ matrix and $B$ an $n\times \rho$ matrix, by randomly selecting  a small, manageable set of columns of $A$ and rows of $B$ under some pre-determined probability rule. Algorithms of this kind are generally pass-efficient, requiring only a constant number of passes over the matrix data for creating samples or sketches. To obtain a good approximation $\hat S$, the choice of probability rule and the column and row scalings are crucial ingredients. In particular, for B{\footnotesize ASIC}M{\footnotesize ATRIX}M{\footnotesize ULTIPLICATION} algorithm in \cite{DKM06}, which is widely discussed in literature, the optimal probability in the sense of minimizing the expected value of $\|AB-\hat S\|^2_F$ is given proportional to the column lengths of A and the row lengths of B. Here $\|\cdot\|_F$ denotes Frobenius norm given by
$$\|A\|_F=\sqrt{\sum_{i,j}A_{(i,j)}^2}=\sqrt{\text{Tr}\big(A^TA\big)}$$
for an arbitrary matrix $A$, where $A_{(i,j)}$ denotes the element of A on the $i$th row and the $j$th column, and $\text{Tr}(\cdot)$ denotes the matrix trace, i.e., the sum of the elements on the main diagonal.

There are other randomised algorithms for matrix multiplication in literature. Random walk are employed in \cite{CL97} and \cite{CL99} to approximate large elements of a matrix product. Low dimensional emmbeddings has be used in \cite{S06} to eliminate data dependence and provide more versatile, linear time pass effecient matrix computation. A sparse representation of linear operators for the approximation of matrix products is studied in \cite{BW08} via  the Nystr\"om extension of a certain positive definite kernel. Subject to the availiablity of information for matrix elements' distributions, an importance sampling strategy based on algorithm B{\footnotesize ASIC}M{\footnotesize ATRIX}M{\footnotesize ULTIPLICATION} is developed in \cite{ESS11} which minimizes the expected value of the variance. For a special case of approximating a Gram matrix $AA^T$, \cite{HI15} presents probabilistic bounds for the 2-norm relative error based on B{\footnotesize ASIC}M{\footnotesize ATRIX}M{\footnotesize ULTIPLICATION}, where the bounds depend on the stable rank or the rank of $A$, but not on the matrix dimensions. As this paper addresses issues arising from algorithm B{\footnotesize ASIC}M{\footnotesize ATRIX}M{\footnotesize ULTIPLICATION}  in \cite{DKM06}, we will now outline our concerns as follows.

The first purpose of this paper is to extend framwork in \cite{DKM06} to a more general setting, say, sampling strategies are no longer restricted to the finest partition of $1$ to $n$. In real problems, columns or rows of a data matrix may be grouped to keep certain property. Thus there is not reason to draw one column (or one row) of the group alone and leave the rest out. This idea is motivated by \cite{WP18}, where the authors consider sketching  the stiffness matrix from finite element Galerkin system on a 3 dimensional domain. The stiffness matrix can be expressed as a symmetric matrix $D^TD$, where $D$ is the concatenated matrix of gradients for linear shape functions with each gradient consisting of $3$ rows. Random sampling should be performed to select all $3$ rows of some linear shape function at the same time in order to preserve at least the positive semidefinite property of $D^TD$.

Section \ref{sec:RamSam} discusses a contruction of $\hat S$ based on an arbitrary partition. Key results like unbiased estimator (Proposition \ref{prop:unbias}) and the optimal probability (Proposition \ref{prop:minsigma}) can be achieved in a similar fashion as \cite{DKM06}. Furthermore, we bound the approximation error in 2-norm,  i.e., the induced norm or the spectrum norm for matrix operators (Euclidean norm for vector) defined by 
$$\|A\|_2: =\sup\Big\{\frac{\|A\mathbf{x}\|_2}{\|\mathbf{x}\|_2}\ \text{for}\ \mathbf{x}\neq 0 \Big\},$$
 in probability sense through noncommutative Bernstein inequality (see Theorem 4.5 in \cite{Mahoney16}), which suggests the optimal probability given in Proposition \ref{prop:minsigma} will guarantee the smallest probability bound. We also examine the spectrum norm of $\hat S$ in Proposition \ref{prop:max1} based on the optimal sampling probability and in Proposition \ref{prop:max} via cumulative binomial distribution if the sampling probability is uniform. At the end of Section \ref{sec:RamSam}, Algorithm \ref{alg:pairwisesampling2} is given for illustrating the random sampling procedure for matrix multiplication from an arbitrary partition.

Furthermore, in terms of the expected value of the squared approximation error in Frobenius norm, the performance are compared in Corollary \ref{cor:largestapproerror} across all sampling strategies with their corresponding optimal sampling distributions, where the sampling procedure based on the finest partition with its optimal probability gives the largest value. This would imply that sampling strategy from the finest partition is least efficient and also suggest a way to enhance the performance of  B{\footnotesize ASIC}M{\footnotesize ATRIX}M{\footnotesize ULTIPLICATION} in \cite{DKM06}, which is discussed in Section \ref{sec:pairwise}. 

Section \ref{sec:pairwise} considers a case when algorithm B{\footnotesize ASIC}M{\footnotesize ATRIX}M{\footnotesize ULTIPLICATION} converges in a slow manner. 
From the point of view of probability, algorithm B{\footnotesize ASIC}M{\footnotesize ATRIX}M{\footnotesize ULTIPLICATION} in one run conducts $c$ independent trials of selecting a single column of $A$ and row of $B$, where the product $AB$ is expressed as the sum of outer products of each single column and row. The distribution of the optimal probability influences the approximation performance with one realization. For example, in case of $B=A^T$, when a few columns of $A$ take much higher weights compared to the rest, it is more likely for those few columns to be drawn within $c$ times sampling from the optimal probability. Thus the outer products of those columns and their transposes
are more likely to be firstly recovered within $c$ trails. Meanwhile, 
the columns with relatively small or tiny weights result in small expected numbers of successful draws within $c$ trails. It implies that it is highly unlikely for those columns to be drawn in one run, thus it is hard to reconstruct the corresponding outer products. But this does not lead to significant deviation as columns with very small or tiny weights only constitute a small proportion of $AA^T$.
On the contrary, when all the columns of $A$ take more or less the same weight, the optimal probability is closed to uniform. Then in one pass, it is not certain that $\|AA^T-\hat S\|_F$ can be small enough with finite $c$ as all columns and rows are nearly equally to be selected. A random sampling strategy on pairwise partition, Algorithm \ref{alg:pairwisesampling1}, is introduced in Section \ref{sec:pairwise} in order to improve the relatively slow convergence to $AB$ caused by the nearly even optimal probability of B{\footnotesize ASIC}M{\footnotesize ATRIX}M{\footnotesize ULTIPLICATION}. The idea is to group columns of $A$ and rows of $B$ in pairs and produce sampling probability by simply adding up the original optimal probabilities, which does not require too much more work. By doing this, we have increased the likelihood of each column being selected and made $\hat S$ more informative within $c$ trails. Algorithm \ref{alg:pairwisesampling1} is preferable due to its better performance in getting a small approximation error in 2-norm with a larger probability (Corollary \ref{cor:argvar}) and bounding the expected value of squared approximation error in Frobenius norm by that from B{\footnotesize ASIC}M{\footnotesize ATRIX}M{\footnotesize ULTIPLICATION} (Corollary \ref{cor:lowerbound}). The best pairing strategy is also discussed among all possible pairwise partitions in terms of approximation error of 2-norm in one run. Besides pairing strategies, Grouping columns of $A$ and rows of $B$ in any other form can be shown to be at least better than B{\footnotesize ASIC}M{\footnotesize ATRIX}M{\footnotesize ULTIPLICATION}. Thus we are quite free to choose any coarser partition. A numerical experiment  is given in Section \ref{sec:ne} to validate results in Section \ref{sec:pairwise} and performances are compared between Algorithm \ref{alg:pairwisesampling1} with the best pairing strategy and B{\footnotesize ASIC}M{\footnotesize ATRIX}M{\footnotesize ULTIPLICATION} with the optimal probability. All the results are summarised in Section \ref{sec:con} and the proofs are postponed to Appendix for a clearer presentation.

\section{An extended framework for random sampling for matrix multiplication}\label{sec:RamSam}
Given an $m\times n$ matrix $A$ and an $n\times \rho$ matrix $B$. The product of the two matrices is of interest. Define the finest partition of $n$ integers as $\Theta$, i.e., $\Theta=\{\{1\},\{2\},\ldots,\{n\}\}$, and sequence the coarser partitions as $(\Theta^{(\kappa)})_\kappa$. Note that each $\Theta^{(\kappa)}$ can be constructed by possible combinations of elements in $\Theta$.

 Suppose we are fixed with a partition $\Theta^\kappa$ with $|\Theta^{(\kappa)}|=k$, then $\Theta^{(\kappa)}=\{\Theta^{(\kappa)}_1,\Theta^{(\kappa)}_2,\ldots,\Theta^{(\kappa)}_k\}$ with each element $\Theta^{(\kappa)}_\ell$ a collection of indices. Denote by $\Theta^{(\kappa)}_{\ell,i}$ the $i$th element of $\Theta^{(\kappa)}_\ell$. Assume there is a sampling probability $p=(p_\ell)_{\ell=1}^k$ with  $\sum_\ell p_\ell=1 $ such that for each $1\leq \ell\leq k$, $\Theta^{(\kappa)}_\ell$ can be drawn with the assigned positive probability $p_\ell$. Note this is equivalent to assigning $p$ to indices of $\Theta^{(\kappa)}$, i.e., $1,2,\ldots, k$. Then if collecting $c$ many index samples  based on $p$ as $\mathbf{r}=\{r_1,r_2,\ldots, r_c\}$ with $r_\ell\in\{1,2,\ldots,k\}$, an approximation to $AB$ can be constructed via the following expression
\begin{equation}\label{hatS}
\hat{S} = \frac 1 c \sum_{i=1}^c \frac{1}{p_{r_i}} A_{(,\Theta^{(\kappa)}_{r_i})}B_{(\Theta^{(\kappa)}_{r_i},)},
\end{equation}
where $A_{(,\Theta^{(\kappa)}_{r_j})}$ is a  $m\times |\Theta^{(\kappa)}_{r_j}| $ matrix with the subscript $(,\Theta^{(\kappa)}_{r_j})$ representing the columns with indices $\Theta^{(\kappa)}_{r_j}$ of this matrix and $B_{(\Theta^{(\kappa)}_{r_i},)}$ a $|\Theta^{(\kappa)}_{r_i}|\times\rho $ matrix with $(\Theta^{(\kappa)}_{r_i},)$ of $B$ representing the rows with indices $\Theta^{(\kappa)}_{r_i}$ of this matrix.
The random sampling procedure in \eqref{hatS} is then equivalent to the following form:
\begin{align}\label{hatSmatrix}
\hat S&= \frac 1 c \sum_{i=1}^c \frac{1}{p_{r_i}} A_{(,\Theta^{(\kappa)}_{r_i})}B_{(\Theta^{(\kappa)}_{r_i},)}=ADC^2D^TB,
\end{align}
where $C$ is a $  \big(\sum_{j=1}^c|\Theta^{(\kappa)}_{r_j}|\big)\times  \big(\sum_{j=1}^c|\Theta^{(\kappa)}_{r_j}|\big)$ diagonal matrix with diagonal entries defined by 
$$C_{\big(\sum_{j=1}^{i-1}|\Theta^{(\kappa)}_{r_j}|+d_i,\sum_{j=1}^{i-1}|\Theta^{(\kappa)}_{r_j}|+d_i\big)}=\frac{1}{\sqrt{cp_{r_i}}}$$ 
for $1\leq i\leq c$ and each $d_i\in\{1,2,\ldots, |\Theta^{(\kappa)}_{r_i}|\}$, and $D$ is a $n\times   \big(\sum_{j=1}^c|\Theta^{(\kappa)}_{r_j}|\big)$ sparse matrix with nonvanishing entries 
$$D_{\big(\Theta^{(\kappa)}_{r_i,d_i},\sum_{j=1}^{i-1}|\Theta^{(\kappa)}_{r_j}|+d_i\big)}=1$$ for each $1\leq i\leq c$ and $1\leq d_i\leq |\Theta^{(\kappa)}_{r_i}|$. $DC^2D^T$ indeed returns a $n\times n$ diagonal matrix with nonnegative entries $(\Theta^{(\kappa)}_{r_i,d_i},\Theta^{(\kappa)}_{r_i,d_i})$ for $1\leq i\leq c$ and $1\leq d_i\leq |\Theta^{(\kappa)}_{r_i}|$ taking value of a ratio of a scalar, which indicates how many times $\Theta^{(\kappa)}_{r_i}$ (or equivalently $r_i$) has been drawn, to $cp_{r_i}$, the expected number of times that $r_i$ has been drawn. 

The estimator $\hat S$ can be shown an unbiased estimator for $AB$ through a probabilistic argument.
\begin{prop}\label{prop:unbias} Given $\Theta^{(\kappa)}$ with $|\Theta^{(\kappa)}|=k$ and its sampling probability $p=(p_\ell)_{\ell=1}^k$. The estimator $\hat S$ constructed in \eqref{hatS} gives an unbiased estimator for $AB$ in the sense of $\mathbb{E}_{p}[\hat S]=AB$, where $\mathbb{E}_{p}$ is expectation under probability $p$.
\end{prop}
An optimal choice for sampling probabilities can be made according to the result below:
\begin{prop}\label{prop:minsigma} Given $\Theta^{(\kappa)}$ with $|\Theta^{(\kappa)}|=k$, and its sampling probability $p=(p_\ell)_{\ell=1}^k$, then the expected value of the squared approximation error in Frobenius norm for approximating $\hat S$ via \eqref{hatS} is 
\begin{align}\label{eqn:var}
\mathbb{E}_{p}[\|AB-\hat S\|^2_F]=  \frac{1}{c}\sum_{\ell=1}^k\frac{1}{p_{\ell}}\big\|A_{(,\Theta^{(\kappa)}_\ell)}B_{(\Theta^{(\kappa)}_\ell,)}\big\|_F^2-\frac{\|AB\|_F^2}{c}.
\end{align}
The optimal sampling probability $p$ in the sense of minimizing $\mathbb{E}_{p}[\|AB-\hat S\|^2_F]$ is given by 
\begin{align}\label{optimalmatrix}
p^{(\kappa)}_\ell := \frac{\big\|A_{(,\Theta^{(\kappa)}_\ell)}B_{(\Theta^{(\kappa)}_\ell,)}\big\|_F}{\sum_{\ell=1}^k\big\|A_{(,\Theta^{(\kappa)}_\ell)}B_{(\Theta^{(\kappa)}_\ell,)}\big\|_F}, \ \ \text{for all}\ \ 1\leq \ell\leq k,
\end{align}
with corresponding expected squared Frobenius-norm error by
\begin{align}\label{eqn:optimalvar} 
V_{p^{(\kappa)}}(AB):=\mathbb{E}_{p^{(\kappa)}}[\|AB-\hat S\|^2_F]= \frac{1}{c}\big(\sum_{\ell=1}^k\big\|A_{(,\Theta^{(\kappa)}_\ell)}B_{(\Theta^{(\kappa)}_\ell,)}\big\|_F\big)^2-\frac{\|AB\|_F^2}{c}.
\end{align}
\end{prop}
Note that when $c\to \infty$, $\mathbb{E}_{p}[\|AB-\hat S\|^2_F]\to 0$ independent of the choice of $p$. 
\begin{remark}
In expression \eqref{optimalmatrix}, a main feature for each coarser partition emerges in the numerator, which we can name as the {\it $\ell$th-element weight of partition $\Theta^{(\kappa)}$}. 
For each $\ell\in\{1,2,\ldots,k\}$, $A_{(,\Theta^{(\kappa)}_\ell)}B_{(\Theta^{(\kappa)}_\ell,)}$, the corresponding element of $AB$, is approximated only through successful draws of $\Theta^{(\kappa)}_\ell$ among partition $\Theta^{(\kappa)}$. Define 
\begin{align}\label{hatSmatrixsingle}
\hat S_{\Theta^{(\kappa)}_\ell}&= \frac 1 c \sum_{i=1}^c \frac{\mathbb{I}_{\ell}(r_i)}{p_{r_i}} A_{(,\Theta^{(\kappa)}_{r_i})}B_{(\Theta^{(\kappa)}_{r_i},)},
\end{align}
where $\mathbb{I}_{\ell}(r_i)$ is the indicator function such that it returns $1$ only if $r_i=\ell$ and returns $0$ otherwise, then the sampling formula \eqref{hatSmatrixsingle} corresponds to $c$ independent binomial trails with success probability $p_\ell$ and failure probability $1-p_\ell$. It can be easily concluded that $\mathbb{E}[\hat S_{\Theta^{(\kappa)}_\ell}]=A_{(,\Theta^{(\kappa)}_\ell)}B_{(\Theta^{(\kappa)}_\ell,)}$ and $\sqrt{\mathbb{E}_p[\|\hat S_{\Theta^{(\kappa)}_\ell}\|_F^2]}=\|A_{(,\Theta^{(\kappa)}_\ell)}B_{(\Theta^{(\kappa)}_\ell,)}\|_F$.
As the current measure is the expected value of the squared approximation error in Frobenius norm, it makes sense that the optimal sampling probability of drawing the $\ell$th index is propotional to $\|A_{(,\Theta^{(\kappa)}_\ell)}B_{(\Theta^{(\kappa)}_\ell,)}\|_F$, the weight of this particular element.  
\end{remark}
Propsition \ref{prop:minsigma} coincides with Lemma 4 in \cite{DKM06} when considering $\Theta$, which gives the exact expected value of approximation error:
\begin{prop}\label{prop:minsigma1}[Lemma 4 in \cite{DKM06}] Given $\Theta$ and its sampling probability $p=(p_\ell)_{\ell=1}^n$, then the expected value of the squared approximation error in Frobenius norm for approximating $\hat S$ in \eqref{hatS} is 
\begin{align}\label{eqn:var1}
\mathbb{E}_{p}[\|AB-\hat S\|^2_F]= \frac{1}{c}\sum_{\ell=1}^n\frac{1}{p_{\ell}}\|A_{(,\ell)}\|_F^2\|B_{(\ell,)}\|_F^2-\frac{\|AB\|_F^2}{c}.
\end{align}
The optimal sampling probability $p$ in the sense of minimizing $\mathbb{E}_{p}[\|AB-\hat S\|^2_F]$ is given by 
\begin{align}\label{optimalmatrix1}
p^{(o)}_\ell := \frac{\|A_{(,\ell)}\|_F\|B_{(\ell,)}\|_F}{\sum_{\ell=1}^n\|A_{(,\ell)}\|_F\|B_{(\ell,)}\|_F}, \ \ \text{for all}\ \ 1\leq \ell\leq n,
\end{align}
with corresponding expected squared Frobenius-norm error bounded by
\begin{align}\label{eqn:optimalvar1} 
V_{p^{(o)}}(AB):=\mathbb{E}_{p^{(o)}}[\|AB-\hat S\|^2_F]= \frac{1}{c}\big(\sum_{\ell=1}^n\|A_{(,\ell)}\|_F\|B_{(\ell,)}\|_F\big)^2-\frac{\|AB\|_F^2}{c}.
\end{align}
\end{prop}
Coarser partition tends to enrich $\hat{S}$ in a quick manner. We can easily reach this by compared \eqref{eqn:optimalvar1} with \eqref{eqn:optimalvar}.
 \begin{corollary}\label{cor:largestapproerror}
To approximate $AB$ with $c$ independent trials, compared to any coarser partition $\Theta^{(\kappa)}$ with their corresponding optimal probability given in \eqref{optimalmatrix}, the finest partition $\Theta$ with its optimal sampling probability gives the largest expected value of the squared approximation error in Frobenius norm. That is, $V_{p^{(\kappa)}}(AB)\leq V_{p^{(o)}}(AB).$
 \end{corollary}
Another way of interpreting this is, the sampling procedure based on the finest partition can be decomposed into two parts, first approximating the coarser partition $\Theta^{(\kappa)}$, and second approximating $AB$ based on the approximation of $\Theta^{(\kappa)}$. Comparing with directly approximating $AB$ via $\Theta^{(\kappa)}$, the sampling procedure via the finest partition is less efficient. Therefore it is not surprising that the approximation via the finest partition is with the largest error. 

Another useful measure is the probabilistic bound for the approximation error in 2-norm, i.e., $\|AB-\hat S\|_2$. It quantifies how good the approximation $\hat S$ is in spectrum norm sense in one go. By definition, spectrum norm measures the largest singular value along while Frobenius norm measures the sum of singular values. An application of noncommutative Bernstein inequality leads to the following upper bound of  $\|AB-\hat S\|_2$:
 \begin{prop}\label{prop:2norm}
Given $\Theta^{(\kappa)}$ with $|\Theta^{(\kappa)}|=k$ and its sampling probability $p=(p_\ell)_{\ell=1}^k$. Assume random sampling procedure $\hat S$ in \eqref{hatS} for approximating $AB$. Then
for each fixed $c$, we have for any $\epsilon>0$ that
\begin{align}\label{eqn:bound2norm}
\mathbb{P}_p\big(\|\hat S-AB\|_2>\epsilon\big)\leq  (m+\rho)\exp\Big(-\frac{c\epsilon^2}{ 2(\|AB\|^2_2+2M\|AB\|_2+ \mathcal{U}_2(p))^2+ \epsilon ( \|AB\|_2 +\mathcal{U}_1(p))}\Big),
\end{align}
where $M:=\sum_{\ell=1}^k\big\|A_{(,\Theta^{(\kappa)}_\ell)}B_{(\Theta^{(\kappa)}_\ell,)}\big\|_F$, and
 $$\mathcal{U}_1(p):=\max_{r}\frac{1}{p_r}\big\|A_{(,\Theta^{(\kappa)}_r)}B_{(\Theta^{(\kappa)}_r,)}\big\|_F\ \ \text{and} \ \ \mathcal{U}_2(p):=\sum_{\ell=1}^k \frac{1}{p_\ell}\big\|A_{(,\Theta^{(\kappa)}_\ell)}B_{(\Theta^{(\kappa)}_\ell,)}\big\|_F^2.$$
\end{prop}

An important observation can be drawn from the right hand side of \eqref{eqn:bound2norm}: with fixed $\Theta^{(\kappa)}$,  among all choices of sampling probabilities, the smallest upper bound would be obtained when both $\mathcal{U}_{1}(p)$ and   $\mathcal{U}_{2}(p)$ are minimized. It can be shown via the method of Lagrange multiplier that, $\mathcal{U}_{1}(p)$ and $\mathcal{U}_{2}(p)$ achieves their minimum when $p=p^{(k)}$ given in \eqref{optimalmatrix}, and the corresponding upper bound can be obtained as an consequence of Porposition \ref{prop:2norm}.
\begin{corollary}\label{cor:2normoptimal}
Given $\Theta^{(\kappa)}$ with $|\Theta^{(\kappa)}|=k$ and its optimal sampling probability $p^{(\kappa)}=(p^{(\kappa)}_\ell)_{\ell=1}^k$. Assume random sampling procedure $\hat S$ in \eqref{hatS} for approximating $AB$. Then
for each fixed $c$, we have for any $\epsilon>0$
\begin{align}\label{eqn:bound2norm'}
\mathbb{P}_p\big(\|\hat S-AB\|_2>\epsilon\big)\leq  (m+\rho)\exp\Big(-\frac{c\epsilon^2}{ 2(\|AB\|_2+M)^2+ \epsilon ( \|AB\|_2 +M)}\Big).
\end{align}
\end{corollary}
Corollary \ref{cor:2normoptimal} suggests that, the optimal sampling probability not only minimises the expected value of the squared approximation error in Frobenius norm but also offers a nicer approximation in spectrum norm sense in one go with large probability.

Another routine result is the upper bound for the spectrum of $\hat S$ from optimal sampling probability, which can be estimated by the sum of element weights of this partition.
\begin{prop}\label{prop:max1}
Given $\Theta^{(\kappa)}$ with $|\Theta^{(\kappa)}|=k$ and its optimal sampling probability $p^{(\kappa)}$. Assume random sampling procedure $\hat S$ in \eqref{hatS} for approximating $AB$. Then we can conclude that
\begin{align}\label{eqn:propmaxeqn1}
\|\hat S\|_2\leq \|\hat S\|_F\leq \sum_{\ell=1}^k\big\|A_{(,\Theta^{(\kappa)}_\ell)}B_{(\Theta^{(\kappa)}_\ell,)}\big\|_F.
\end{align}
\end{prop}
Proposition \ref{prop:max1} does not relate $\|\hat S\|_2$ directly to $\|A\|_2\|B\|_2$ or $\|AB\|_2$. As a special case, this can be achieved if $p$ is uniform:
\begin{prop}\label{prop:max}
Given $\Theta^{(\kappa)}$ with $|\Theta^{(\kappa)}|=k$ and its sampling probability $p=(p_\ell)_{\ell=1}^k$, $p_\ell=\frac{1}{k}$. Assume random sampling procedure $\hat S$ in \eqref{hatS} for approximating $AB$. Further assume $100\leq k^{c-1}$. Then
for each fixed $c$, choose $s_c\in \mathbb{N}$ such that 
\begin{align}\label{eqn:boundnc}
s^{(k)}_c=\min_{2\leq s\leq c}\big\{s: s\geq 100c\big(1-F_{\text{b}}\big(s-2;c-1,\frac{1}{k}\big)\big)\big\},
\end{align}
 where $F_{\text{b}}\big(s;N,\xi)$ is the cumulative binomial distribution function with $N$ the total number of trials, $\xi$ the probability of success and $s$ the number of success interested; then with probability at least $0.99$, 
\begin{align}\label{eqn:propmaxeqn}
\|\hat S\|_2\leq \frac{k(s^{(k)}_c-1)}{c}\|A\|_2\|B\|_2.
\end{align}
\end{prop}
\begin{remark} 
\begin{enumerate}
\item The assumption $100\leq k^{c-1}$ ensures the valid upper bound $c$ for $s^{(k)}_c$ in Eqn. \eqref{eqn:boundnc}. This can be seen by simply replacing $s$ with $c$ in the inequality of \eqref{eqn:boundnc} and then using the fact that $1-F_{\text{b}}\big(c-2;c-1,\frac{1}{k}\big)=F_{\text{b}}\big(0;c-1,\frac{k-1}{k}\big)=\big(\frac 1 k \big)^{c-1}$. 
\item Note that condition \eqref{eqn:boundnc} is not restrictive. For instance, pick $c=500$ and $k=2000$, then $s^{(k)}_c=3$. Usually $s^{(k)}_c$ does not have a closed form, however, one can start trying from very small number, say, $m=2$.
\item When $B=A^T$, we have $\|\hat S\|_2\leq \frac{k(s^{(k)}_c-1)}{c}\|AA^T\|_2$.
\end{enumerate} 
\end{remark}
Finally, the random sampling procedure for matrix multiplication based on an arbitrary partition is given in Algorithm \ref{alg:pairwisesampling2}. 
\begin{algorithm}
\caption{Random sampling for matrix multiplication based on an arbitrary partition.}
\label{alg:pairwisesampling2}
\begin{algorithmic}[1]
\State \textbf{input:} $A$ and $B$, the targeted matrices for doing matrix multiplication;
\State \hspace{1.1cm} $c$, the sample size;
\State \hspace{1.1cm} $\Theta^{(\kappa)}$ with $|\Theta^{(\kappa)}|=k$, the targeted partition of collections of indices;
\State \hspace{1.1cm} $(p_{\ell})_{\ell=1}^k$, the sampling probability  to $\Theta^{(\kappa)}$.
\State \textbf{output:} $\hat S$, the sketched version of $AB$.
\State \textbf{initialization}: $\hat S$.
\For {$i=1\cdots c$ (the number of iterations)}
\State random pick an index $r_i$ from $1$ to $k$ based on probabilities $(p_{\ell})_{\ell=1}^k$;
\State set $\hat S=\hat S+ \frac{1}{cp_{r_i}}A_{(,\Theta^{(\kappa)}_{r_i})}B_{(\Theta^{(\kappa)}_{r_i},)}$,
\EndFor
\State \textbf{return}: $\hat S$.
\end{algorithmic}
\end{algorithm}
\section{A random sampling strategy for matrix multiplication based on pairwise partition}\label{sec:pairwise}
Random sampling based on a coarser partition discussed in Section \ref{sec:RamSam} also suggests a way to acceralate the random sampling procedure based on the finest partition, i.e., algorithm B{\footnotesize ASIC}M{\footnotesize ATRIX}M{\footnotesize ULTIPLICATION} in \cite{DKM06}. Though converging with order $\frac{1}{2}$, which is intrinsically given by Monte Carlo method, different choices of sampling probabilities may still affect slightly the convergence speed of B{\footnotesize ASIC}M{\footnotesize ATRIX}M{\footnotesize ULTIPLICATION}. Needless to say, with the optimal probability, which minimizes the expected bound of $\|AB-\hat S\|^2_F$, B{\footnotesize ASIC}M{\footnotesize ATRIX}M{\footnotesize ULTIPLICATION} will give the best performance in terms of convergence speed. In practice, however, we usually attempt to approximate $AB$ in one run. As discussed in Section \ref{sec:intro}, when the optimal sampling probability is closed to uniform, B{\footnotesize ASIC}M{\footnotesize ATRIX}M{\footnotesize ULTIPLICATION} may not provide an accurate solution when $c$ is not large enough. With $c$ fixed, an idea of grouping columns of $A$ and rows of $B$ in pairs (or in any pattern) in order to enrich the sampling procedure is introduced in this Section, which can be shown to have a better performance in both average and pathwise sense.

Now assume that $n$ is an even number. Suppose we have assigned the optimal random sampling probability for the finest partition $\Theta$, i.e., 
the probability of drawing the $\ell$th index among $1,2,\ldots,n$ is given as $p^{(o)}_\ell$ according to \eqref{optimalmatrix1} for each $\ell\in\{1,2,\ldots,n\}$. 
Then we may define a one-to-one correspondence $I:\{1,2\ldots,n\}\to \{1,2\ldots,n\}$ according to some sorting rule, and reorder the indices as $\{I_1, I_2, \ldots, I_n\}$ where $I_\ell$ is short for $I(\ell)$. 
 
Paring indices $I_{2j-1}$ and $I_{2j}$ for $j\in\{1,2,\ldots,\frac{n}{2}\}$ yields
\begin{align}\label{eqn:pairing}
\Theta^{(\text{pair})}=\big\{\{I_1,I_2\},\{I_3,I_{4}\},\ldots,\{I_{n-1},I_{n}\}\big\}
\end{align}
and simply assigning the sampling probability 
\begin{align}\label{eqn:ppair}
 p^{\text{(pair)}}_{j}:=p^{(o)}_{I_{2j-1}}+p^{(o)}_{I_{2j}}
\end{align} 
to $\Theta^{(\text{pair})}_j=\{I_{2j-1},I_{2j}\}$ (which is indeed to $j$) and getting $\hat S$ as in \eqref{hatSmatrix}. The algorithm is illustrated in Algrorithm \ref{alg:pairwisesampling1}.
\begin{algorithm}
\caption{Random sampling for matrix multiplication based on pairwise partition.}
\label{alg:pairwisesampling1}
\begin{algorithmic}[1]
\State \textbf{input:} $A$ and $B$, the targeted matrices for doing matrix multiplication;
\State \hspace{1.1cm} $c$, sampler size;
\State \hspace{1.1cm} $p^{(o)}$, optimal sampling probabilities to $\Theta$.
\State \textbf{output:} $\hat S$, the sketched version of $AB$.
\State \textbf{initialization}:  $I$, a vector of size $n$ to store indices; $\hat S$; $p^{\text{(pair)}}$.
\State reorder $\{1,2,\ldots,n\}$ and record in $I$ (like in \eqref{eqn:pairing});
\State get $p^{\text{(pair)}}$ via \eqref{eqn:ppair};
\For {$i=1\cdots c$ (the number of iterations)}
\State random pick an index $r_i$ from $1$ to $\frac{n}{2}$ based on probabilities $(p^{\text{(pair)}}_\ell)_{\ell=1}^{\frac{n}{2}}$;
\State set $\hat S=\hat S+ \frac{1}{cp^{\text{(pair)}}_{r_i}}(A_{(,I_{2r_i-1})}+A_{(,I_{2r_i})})(B_{(I_{2r_i-1},)}+B_{(I_{2r_i},)})$
\EndFor
\State \textbf{return}: $\hat S$.
\end{algorithmic}
\end{algorithm}

Note that $\Theta^{(\text{pair})}$ can be constructed from any pairwise partition and the sampling probability can be defined accordingly though \eqref{eqn:ppair}. Here we list four pairing strategies:
\begin{description}
\item[Enhanced pairwise partition.] That is, we may sort $\big(p^{(o)}_\ell\big)_{\ell=1}^n$ in an ascending order, and define $(I_\ell)_{\ell=1}^n$ with $I_\ell\in\{1,2,\ldots,n\}$ such that $p^{(o)}_{I_1}\leq p^{(o)}_{I_2}\leq \ldots\leq p^{(o)}_{I_n}$. Equivalently, pairing the index with the smallest probability and the index with the second smallest probability, paring the index with the third smallest probability and the index with the fourth smallest probability, so on until we are left with the last pair.
\item[Random pairwise partition.] That is, we may sort $\{1,2,\ldots,n\}$ in a random order (random permutation) and define the new sequence as $I$.
\item[Balanced pairwise partition.] That is, we may first sort $\big(p^{(o)}_\ell\big)_{\ell=1}^n$ in an ascending order, and define $(\hat{I}_\ell)_{\ell=1}^n$ with $\hat{I}_\ell\in\{1,2,\ldots,n\}$ such that $p^{(o)}_{\hat{I}_1}\leq p^{(o)}_{\hat{I}_2}\leq \ldots\leq p^{(o)}_{\hat{I}_n}$; and then define $(I_\ell)_{\ell=1}^n$ with $I_\ell\in\{1,2,\ldots,n\}$ such that $I_{2j-1}=\hat{I}_{n-j}$ and $I_{2j}=\hat{I}_{j}$ for $j\in\{1,2,\ldots,\frac{n}{2}\}$. Equivalently, pairing the index with the largest probability and the index with the smallest probability, paring the index with the second largest probability and the index with the second smallest probability, so on until we are left with the last pair.
\item[Simple pairwise partition.] That is, we may simply pair the first column with the second one, the third column with the fourth one, and so on; equivalently, set $I(\ell)=\ell$ for all $\ell\in\{1,2,\ldots,n\}$. 
\end{description}

Some approximation errors are measured to examine the performance of Algorithm \ref{alg:pairwisesampling1}.
\begin{corollary}\label{cor:argvar} Fixed with sample size $c$. The expected value of the squared approximation error in Frobenius norm from Algorithm \ref{alg:pairwisesampling1} is
\begin{align}\label{eqn:varcor} 
\mathbb{E}_{p^{(\text{pair})}}[\|AB-\hat S\|^2_F]\leq V_{p^{(o)}}(AB).
\end{align}
\end{corollary}
As Algorithm \ref{alg:pairwisesampling1} is intrinsically based on $\Theta$ with optimal probability $p^{(o)}$ and the corresponding sampling probability has been increased in \eqref{eqn:ppair}, it is not surprising that the approximation error bound in Frobenius norm from Algorithm \ref{alg:pairwisesampling1} is bounded by that from algorithm B{\footnotesize ASIC}M{\footnotesize ATRIX}M{\footnotesize ULTIPLICATION} with optimal probability $p^{(o)}$.

Note that Corollary \ref{cor:argvar} is valid for Algorithm \ref{alg:pairwisesampling1} with arbitrary pairing strategy.  We cannot distinguish the performances of Algorithm \ref{alg:pairwisesampling1} with different pairing strategies. A further investigation following the same argument of Proposition \ref{prop:2norm} suggests that, among all pairing strategies, Algorithm \ref{alg:pairwisesampling1} with {\it enhanced pairwise partition} gives the best approximation in probability sense.
\begin{corollary}\label{cor:lowerbound} Fixed with sample size $c$. Compared with B{\footnotesize ASIC}M{\footnotesize ATRIX}M{\footnotesize ULTIPLICATION},  Algorithm \ref{alg:pairwisesampling1} has a smaller upper bound for the probability $\mathbb{P}_{p^{(o)}}(\|AB-\hat S\|_2> \epsilon)$, where $\hat S$ is a generic symbol for the resulted approximation of $AB$ either via B{\footnotesize ASIC}M{\footnotesize ATRIX}M{\footnotesize ULTIPLICATION} or Algorithm \ref{alg:pairwisesampling1}. Furthermore, among all possible pairing strategies, Algorithm \ref{alg:pairwisesampling1} with enhanced pairwise partition gives the smallest upper bound.
\end{corollary}
\begin{remark}The proof of Corollary \ref{cor:lowerbound} also suggests that, Algorithm \ref{alg:pairwisesampling1} with any pairing strategy would give a smaller upper bound for the probability $\mathbb{P}_{p^{(o)}}(\|AB-\hat S\|_2> \epsilon)$ than B{\footnotesize ASIC}M{\footnotesize ATRIX}M{\footnotesize ULTIPLICATION}. This might be due to the fact that pairing indices together will increase the likelihood of each index and thus each column and row being drawn. For instance, \eqref{eqn:ppair} increases the probability for both indices $I_{2\ell-1}$ and $I_{2\ell}$ from $p^{(o)}_{I_{2\ell-1}}$ and $p^{(o)}_{I_{2\ell}}$ respectively to $p^{(\text{pair})}_\ell$ at the same time. Within $c$ independent trials, more information about $AB$ will be brought in through Algorithm \ref{alg:pairwisesampling1} with pairing strategy, thus $\hat S$ will be more informative and gives a better approximation in one run with larger probability.
\end{remark}
Indeed Corollary \ref{cor:argvar} can be extended to any coarser partition. Given $\Theta^{(\kappa)}$ with $|\Theta^{(\kappa)}|=k$. Define sampling distribution $p^{(\kappa,o)}$ from the optimal probability of the finest partition $\Theta$ as
\begin{align}\label{eqn:pkappao}
p^{(\kappa,o)}_\ell=\sum_{d_\ell=1}^{|\Theta^{(\kappa)}_\ell |} p^{(o)}_{\Theta^{(\kappa)}_{\ell,d_\ell}}, \ \ \text{for}\ \ \ell\in\{1,2,\ldots,k\}.
\end{align}
Then we can implement approximation procedure for $AB$ as described in \eqref{hatSmatrix} and conclude in the same fashion as Corollary \ref{cor:argvar}:
\begin{corollary}\label{cor:argvar1} Given $\Theta^{(\kappa)}$ with $|\Theta^{(\kappa)}|=k$ and the sampling distribution $p^{(\kappa,o)}$  as defined in \eqref{eqn:pkappao}. If fixed with sample size $c$, then the expected value of the squared approximation error in Frobenius norm from  \eqref{hatSmatrix} satisfies
\begin{align}\label{eqn:varcor1} 
\mathbb{E}_{p^{(\kappa,o)}}[\|AB-\hat S\|^2_F]\leq V_{p^{(o)}}(AB).
\end{align}
\end{corollary}
This would imply that the sampling procedure based on an arbitrary coarser partition with the assigned sampling probability \eqref{eqn:pkappao} would provide a nicer approximation for $AB$ regarding the approximation error in Frobenius norm. The same argument mentioned after Corollary \ref{cor:largestapproerror} applies here as well. The sampling procedure based on the finest partition $\Theta$ can be treated as first conducting a sampling procedure towards $\Theta^{(\kappa)}$, and then approximating $AB$ based on the approximation of partition $\Theta^{(\kappa)}$. Therefore the approximation error based on partition $\Theta^{(\kappa)}$ is not bigger than the one based on the finest partition. Corollary \ref{cor:argvar1} allows the flexibility of choosing any coarser partition for approximating $AB$ and Algorithm \ref{alg:pairwisesampling1} needs to be tailored accordingly. 

\subsection{An experiment}\label{sec:ne}
In this part, we compare the performances of Algorithm \ref{alg:pairwisesampling1} and B{\footnotesize ASIC}M{\footnotesize ATRIX}M{\footnotesize ULTIPLICATION} in \cite{DKM06}, which only works on the finest partition $\Theta$. 

We generate a random matrix $A$ from standard uniform distribution with sizes $100\times 2000$. That is, we independently draw elements of $A$ from standard uniform distribution. This would give those columns with relatively equal column norm, which is indeed 2-norm of this column. 


We then implement algorithm B{\footnotesize ASIC}M{\footnotesize ATRIX}M{\footnotesize ULTIPLICATION} and Algorithm \ref{alg:pairwisesampling1} on {\it enhanced pariwise partition} with $c$ varying from $1000$ to $3000$ to approximate $AA^T$ and compare their performances by calculating the corresponding relative Frobenius norm error via a Monte Carlo simulation with 1000 independent samples, i.e., $\frac{1}{\|AA^T\|_F}\mathbb{E}[\|AA^T-\hat S\|_F]$. 

\begin{figure}
  \begin{center}
      \includegraphics[width=0.6\textwidth]{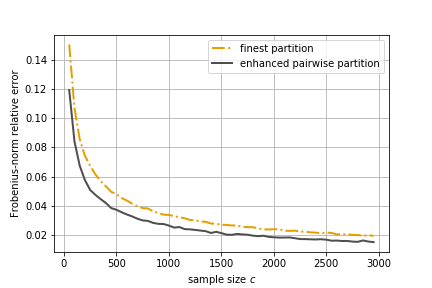}
                    \caption{\small Approximating $AA^T$:
      sample size $c$ versus Frobenius-norm relative error; orange dashed line represents error curve from algorithm B{\footnotesize ASIC}M{\footnotesize ATRIX}M{\footnotesize ULTIPLICATION} based on the finest partition, solid line represents error curve from Algorithm \ref{alg:pairwisesampling1} based on enhanced pairwise partition.
     \label{fig1}} 
  \end{center}
\end{figure}


Figure \ref{fig1} demonstrates the error trends for both algorithms. We can clearly see that the relative error generated from Algorithm \ref{alg:pairwisesampling1}, the solid line, is always below the dashed line, the one from B{\footnotesize ASIC}M{\footnotesize ATRIX}M{\footnotesize ULTIPLICATION}. This may verify the discussion in Section \ref{sec:pairwise} that the expected relative error from Algorithm \ref{alg:pairwisesampling1} is bounded by that from B{\footnotesize ASIC}M{\footnotesize ATRIX}M{\footnotesize ULTIPLICATION}.

To investigate the results, we collect the statistics for $p^{(o)}$ and $p^{\text{(pair)}}$ in Table \ref{tab1}. With the sampling probability based on enhanced pairwise partition, we increase the minimal value of the sampling probability while keeping almost the same the maximal value. Increasing the sampling probability results in more successful draws from each columns to some extent.


\begin{table}[ht]
\caption{   \label{tab1} Statistics for $p^{(o)}$ and $p^{\text{(pair)}}$.} 
\begin{center}
\begin{tabular}{ccc|ccc}
 \hline  \hline
$\max(p^{(o)})$ & $\text{mean}(p^{(o)})$ & $\min(p^{(o)})$ &$\max(p^{\text{(pair)}})$ & $\text{mean}(p^{\text{(pair)}})$ & $\min(p^{\text{(pair)}})$ \\
  \hline
0.00065 &0.00050& 0.00033& 0.00131 &0.00100 & 0.00070\\
\hline  \hline
\end{tabular}
\end{center}
\end{table}


\begin{figure}
\centering
\subfigure[a][$c=1000$]{\includegraphics[width=0.47\textwidth]{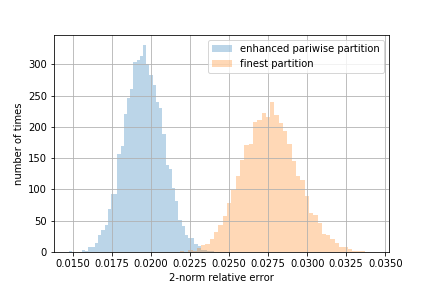}}\qquad
\subfigure[b][$c=3000$]{\includegraphics[width=0.47\textwidth]{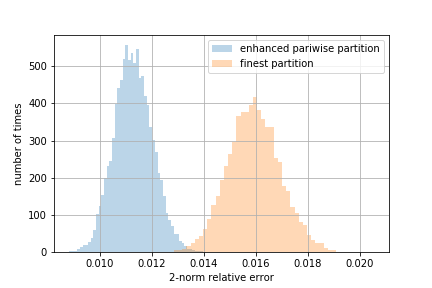}}\\
\caption{Histograms for approximation error: 2-norm relative error versus counts; the left figure is for sampler size $c=1000$, and the right figure is for sampler size $c=3000$; in both figures light blue shape indicates histogram from Algorithm \ref{alg:pairwisesampling1} based on enhancd pairwise partition and light orange shape indicates histogram from algorithm B{\footnotesize ASIC}M{\footnotesize ATRIX}M{\footnotesize ULTIPLICATION} based on the finest partition.\label{fig2}}
\end{figure}

Besides, we implement Algorithm \ref{alg:pairwisesampling1} and algorithm B{\footnotesize ASIC}M{\footnotesize ATRIX}M{\footnotesize ULTIPLICATION} $50000$ times in order to get shapes for distributions of 2-norm relative error under different sample sizes, $c=1000$, which is smaller than $n=2000$, and $c=3000$, bigger than $n$. In both cases, we can see from these histograms that the error distribution from Algorithm \ref{alg:pairwisesampling1}, the shape in light blue, is left peaked with higher maximum than the shape in light orange, the error distribution from algorithm B{\footnotesize ASIC}M{\footnotesize ATRIX}M{\footnotesize ULTIPLICATION}. This may imply that, it is more likely for Algorithm \ref{alg:pairwisesampling1} to approximate $AA^T$ via a smaller 2-norm error than algorithm B{\footnotesize ASIC}M{\footnotesize ATRIX}M{\footnotesize ULTIPLICATION}. In addition, the shapes for $c=3000$ as in Figure \ref{fig2} (b) are more closed to $0$. This consists with the fact that the approximation error vanishes as $c\to \infty$. 
\section{Conclusion}\label{sec:con}
This paper extends the framework of random sampling for matrix multiplication in \cite{DKM06} to a coarser partition. A detailed discussion on the performance is given through measuring the approximation error in Frobenius norm and 2-norm, see Proposition \ref{prop:unbias}, \ref{prop:minsigma} and \ref{prop:2norm}. In practice, B{\footnotesize ASIC}M{\footnotesize ATRIX}M{\footnotesize ULTIPLICATION} in \cite{DKM06} with the optimal probability that is more closed to a uniform distribution may lead to a relatively slow convergence. To improve the performance, Algorithm \ref{alg:pairwisesampling1} with parwise partition is proposed as a complement to B{\footnotesize ASIC}M{\footnotesize ATRIX}M{\footnotesize ULTIPLICATION}. In particular, the new algorithm increases the likelihood of getting a small approximation error in 2-norm and reduces the squared approximation error in Frobenious norm to be bounded by that from algorithm B{\footnotesize ASIC}M{\footnotesize ATRIX}M{\footnotesize ULTIPLICATION}. Furthermore, Algorithm \ref{alg:pairwisesampling1} can be modified to accommodate any coarser partition, which will give a smaller squared approximation error in Frobenious norm compared with B{\footnotesize ASIC}M{\footnotesize ATRIX}M{\footnotesize ULTIPLICATION}.
\section*{Acknowledgement}
 The author would gratefully acknowledge financial support by EPRSC EP/R041431/1.
 \bibliographystyle{alpha}

\section*{Appendix}
\begin{proof}[Proof of Proposition \ref{prop:unbias}] The claim can be shown through the linear property of expectation and definition of expectation as follows
\begin{align*}
\mathbb{E}_{p}[\hat S]&=\mathbb{E}_{p}\Big[ \frac 1 c \sum_{i=1}^c \frac{1}{p_{r_i}} A_{(,\Theta^{(\kappa)}_{r_i})}B_{(\Theta^{(\kappa)}_{r_i},)}\Big]= \frac 1 c\sum_{i=1}^c\mathbb{E}_{p}\Big[  \frac{1}{p_{r_i}} A_{(,\Theta^{(\kappa)}_{r_i})}B_{(\Theta^{(\kappa)}_{r_i},)}\Big]\\
&= \frac 1 c\sum_{i=1}^c\sum_{\ell=1}^k \frac{p_{r_\ell}}{p_{r_\ell}} A_{(,\Theta^{(\kappa)}_{r_\ell})}B_{(\Theta^{(\kappa)}_{r_\ell},)}=\sum_{\ell=1}^k A_{(,\Theta^{(\kappa)}_{r_\ell})}B_{(\Theta^{(\kappa)}_{r_\ell},)}=AB.
\end{align*}
\end{proof}
\begin{proof}[Proof of Proposition \ref{prop:minsigma}]
 First of all, note that for all $1\leq h_1\leq m,1\leq h_2\leq \rho$ we have
$$\hat S_{(h_1,h_2)}=  \sum_{i=1}^c \frac{1}{cp_{r_i}} A_{(h_1,\Theta^{(\kappa)}_{r_i})}B_{(\Theta^{(\kappa)}_{r_i},h_2)}.$$ 
A simple consequence from $\mathbb{E}_{p}[\hat S]=AB$ is that $\mathbb{E}_{p}[\hat S_{(h_1,h_2)}]=(AB)_{(h_1,h_2)}$. We then have that
$$\mathbb{E}_p[\|AB-\hat S\|_F^2]=\sum_{h_1=1}^m \sum_{h_2=1}^\rho \mathbb{E}_p[(AB-\hat S)^2_{(h_1,h_2)}]=\sum_{h_1=1}^m \sum_{h_2=1}^\rho \textbf{Var}_{p}[\hat S_{(h_1,h_2)}].$$
For each pair $(h_1,h_2)$, we have that
\begin{align*}
&\textbf{Var}_{p}[\hat S_{(h_1,h_2)}]=\textbf{Var}_{p}\big[ \sum_{i=1}^c \frac{1}{cp_{r_i}} A_{(h_1,\Theta^{(\kappa)}_{r_i})}B_{(\Theta^{(\kappa)}_{r_i},h_2)}\big]= \sum_{i=1}^c\frac{1}{c^2}\textbf{Var}_{p}\big[ \frac{1}{p_{r_i}} A_{(h_1,\Theta^{(\kappa)}_{r_i})}B_{(\Theta^{(\kappa)}_{r_i},h_2)}\big]\\
&\quad =\sum_{i=1}^c\frac{1}{c^2}\mathbb{E}_{p}\big[( \frac{1}{p_{r_i}} A_{(h_1,\Theta^{(\kappa)}_{r_i})}B_{(\Theta^{(\kappa)}_{r_i},h_2)}\big)^2\big]-\sum_{i=1}^c\frac{1}{c^2}\mathbb{E}_{p}\big[( \frac{1}{p_{r_i}} A_{(h_1,\Theta^{(\kappa)}_{r_i})}B_{(\Theta^{(\kappa)}_{r_i},h_2)}\big)\big]^2\\
&\quad =\sum_{i=1}^c\frac{1}{c^2}\sum_{\ell=1}^kp_{\ell}( \frac{1}{p_{\ell}} A_{(h_1,\Theta^{(\kappa)}_\ell)}B_{(\Theta^{(\kappa)}_\ell,h_2)}\big)^2-\sum_{i=1}^c\frac{1}{c^2}\big(\sum_{\ell=1}^kp_{\ell}\frac{1}{p_\ell} A_{(h_1,\Theta^{(\kappa)}_\ell)}B_{(\Theta^{(\kappa)}_\ell,h_2)}\big)^2\\
&\quad=\frac{1}{c}\sum_{\ell=1}^k\frac{1}{p_{\ell}} \big(A_{(h_1,\Theta^{(\kappa)}_\ell)}B_{(\Theta^{(\kappa)}_\ell,h_2)}\big)^2-\frac{1}{c}(AB)_{(h_1,h_2)}^2,
\end{align*}
where the second equality holds because of independence. This in turn implies that
\begin{align*}
&\sum_{h_1=1}^m \sum_{h_2=1}^\rho\textbf{Var}_{p}[\hat S_{(h_1,h_2)}] =\frac{1}{c} \sum_{h_1=1}^m\sum_{h_2=1}^\rho\Big( \sum_{\ell=1}^k\frac{1}{p_{\ell}} \big(A_{(h_1,\Theta^{(\kappa)}_\ell)}B_{(\Theta^{(\kappa)}_\ell,h_2)}\big)^2-(AB)_{(h_1,h_2)}^2\Big)\\
&\quad = \frac{1}{c}\sum_{\ell=1}^k\frac{1}{p_{\ell}}\|A_{(,\Theta^{(\kappa)}_\ell)}B_{(\Theta^{(\kappa)}_\ell,)}\|_F^2-\frac{\|AB\|_F^2}{c}: = \mathcal{U}(p).
\end{align*}
Note that $\mathcal{U}(p)$ characterizes the dependence of $\mathbb{E}_p[\|AB-\hat S\|_F^2]$ on $p$. In order to optimize  $\mathcal{U}(p)$, a Lagrange multiplier $\lambda$ is introduced with Lagrange function
\begin{align*}
\mathcal{L}(p;\lambda)=- \mathcal{U}(p)-\lambda\Big(\sum_{\ell=1}^k p_{\ell}-1\Big).
\end{align*}
 The method of Lagrange multiplier returns 
 \begin{align*}
p^{(\kappa)}_\ell = \frac{\big\|A_{(,\Theta^{(\kappa)}_\ell)}B_{(\Theta^{(\kappa)}_\ell,)}\big\|_F}{\sum_{\ell=1}^k\big\|A_{(,\Theta^{(\kappa)}_\ell)}B_{(\Theta^{(\kappa)}_\ell,)}\big\|_F}.
\end{align*}
Plugging optimal expression of $p$ into $\mathcal{U}(p)$ suggests that
\begin{align*} 
\mathbb{E}_{p^{(\kappa)}}[\|AB-{\hat S}\|^2_F]= \frac{1}{c}\big(\sum_{\ell=1}^k\big\|A_{(,\Theta^{(\kappa)}_\ell)}B_{(\Theta^{(\kappa)}_\ell,)}\big\|_F\big)^2-\frac{\|AB\|_F^2}{c}.
\end{align*}
\end{proof}

\begin{proof}[Proof of Corollary \ref{cor:largestapproerror}]
Directly subtracting \eqref{eqn:optimalvar} from \eqref{eqn:optimalvar1} yields that
\begin{align*}
&V_{p^{(o)}}(AB)-\mathbb{E}_{p^{(\kappa)}}[\|AB-{\hat S}\|^2_F]=\frac{1}{c}\big(\sum_{\ell=1}^n\|A_{(,\ell)}B_{(\ell,)}\|_F\big)^2-\frac{1}{c}\big(\sum_{\ell=1}^k\big\|A_{(,\Theta^{(\kappa)}_\ell)}B_{(\Theta^{(\kappa)}_\ell,)}\big\|_F\big)^2\\
&= \frac{1}{c}\Big(\sum_{\ell=1}^n\|A_{(,\ell)}B_{(\ell,)}\|_F-\sum_{\ell=1}^k\big\|A_{(,\Theta^{(\kappa)}_\ell)}B_{(\Theta^{(\kappa)}_\ell,)}\big\|_F\Big)\Big(\sum_{\ell=1}^n\|A_{(,\ell)}B_{(\ell,)}\|_F+\sum_{\ell=1}^k\big\|A_{(,\Theta^{(\kappa)}_\ell)}B_{(\Theta^{(\kappa)}_\ell,)}\big\|_F\Big).
\end{align*}
Note that
\begin{align*}
\sum_{\ell=1}^k\big\|A_{(,\Theta^{(\kappa)}_\ell)}B_{(\Theta^{(\kappa)}_\ell,)}\big\|_F=\sum_{\ell=1}^k\Big\|\sum_{\ell_i=1}^{|\Theta^{(\kappa)}_{\ell,}|} A_{(,\Theta^{(\kappa)}_{\ell,\ell_i})}B_{(\Theta^{(\kappa)}_{\ell,\ell_i},)}\Big\|_F\leq \sum_{\ell=1}^n\|A_{(,\ell)}B_{(\ell,)}\|_F
\end{align*}
by the property of norm. This implies that $V_{p^{(o)}}(AB)-\mathbb{E}_{p^{(\kappa)}}[\|AB-{\hat S}\|^2_F]\geq 0$.
\end{proof}

\begin{proof}[Proof of Proposition \ref{prop:2norm}]
Define a random index $r\in\{1,2,\ldots,k\}$, with $\mathbb{P}_p(r=\ell)=p_\ell$ and thus a random rectangular matrix $Y=\frac{1}{p_r}A_{(,\Theta^{(\kappa)}_r)}B_{(\Theta^{(\kappa)}_r,)}$. $Y-AB$ can be shown as zero mean matrix:
$$\mathbb{E}_{p}[Y]-AB=\sum_{\ell=1}^k p_\ell \frac{1}{p_\ell}A_{(,\Theta^{(\kappa)}_\ell)}B_{(\Theta^{(\kappa)}_\ell,)}-AB=0.$$
Besides, 
\begin{align*}
&\|AB-Y\|_2=\big\|AB-\frac{1}{p_r}A_{(,\Theta^{(\kappa)}_r)}B_{(\Theta^{(\kappa)}_r,)}\big\|_2\\
&\quad \leq \|AB\|_2 +\max_{r}\frac{1}{p_r}\big\|A_{(,\Theta^{(\kappa)}_r)}B_{(\Theta^{(\kappa)}_r,)}\big\|_F= \|AB\|_2 +\mathcal{U}_1(p).
\end{align*}
To apply with noncommutative Bernstein inequality, we also need to estimate
\begin{align*}
&\|\mathbb{E}_{p}[(Y-AB)^T(Y-AB)]\|_2\\
&\quad =\big\| \sum_{\ell=1}^k p_\ell \big[\big(\frac{1}{p_\ell}A_{(,\Theta^{(\kappa)}_\ell)}B_{(\Theta^{(\kappa)}_\ell,)}-AB\big)^T\big(\frac{1}{p_\ell}A_{(,\Theta^{(\kappa)}_\ell)}B_{(\Theta^{(\kappa)}_\ell,)}-AB\big)\big]\big\|_2\\
&\quad \leq  \sum_{\ell=1}^k p_\ell \big( \|AB\|_2 +\frac{1}{p_\ell}\big\|A_{(,\Theta^{(\kappa)}_\ell)}B_{(\Theta^{(\kappa)}_\ell,)}\big\|_F\big)^2 \\
&\quad =\|AB\|^2_2+2\|AB\|_2 \sum_{\ell=1}^k\big\|A_{(,\Theta^{(\kappa)}_\ell)}B_{(\Theta^{(\kappa)}_\ell,)}\big\|_F + \sum_{\ell=1}^k \frac{1}{p_\ell}\big\|A_{(,\Theta^{(\kappa)}_\ell)}B_{(\Theta^{(\kappa)}_\ell,)}\big\|_F^2\\
&\quad=\|AB\|^2_2+2M\|AB\|_2+ \mathcal{U}_2(p) ,
\end{align*}
and the same bound holds for $ \|\mathbb{E}_{p}[(Y-AB)(Y-AB)^T]\|_2$. Now sample $c$ copies of $Y$. Let $\bar Y=\sum_{j=1}^c\frac{Y_j}{c}$. For any $\epsilon>0$, applying with noncommutative Bernstein inequality yields that
\begin{align*}
\mathbb{P}_p\big(\|\bar Y-AB\|_2>\epsilon\big)\leq (m+\rho)\exp\Big(-\frac{c\epsilon^2}{ 2(\|AB\|^2_2+2M\|AB\|_2+ \mathcal{U}_2(p))^2+ \epsilon ( \|AB\|_2 +\mathcal{U}_1(p))}\Big).
\end{align*}
\end{proof}
\begin{proof}[Proof of Proposition \ref{prop:max1}]
By the property of norm we have that
\begin{align*}
&\|\hat S\|_F\leq \frac{1}{c}\sum_{i=1}^c \frac{1}{p^{(\kappa)}_{r_i}}\|A_{(,\Theta^{(\kappa)}_{r_i})}B_{(\Theta^{(\kappa)}_{r_i},)}\|_F=\frac{1}{c}\sum_{i=1}^c\sum_{\ell=1}^k \|A_{(,\Theta^{(\kappa)}_\ell)}B_{(\Theta^{(\kappa)}_\ell,)}\|_F=\sum_{\ell=1}^k \|A_{(,\Theta^{(\kappa)}_\ell)}B_{(\Theta^{(\kappa)}_\ell,)}\|_F,
\end{align*}
the first equality is true by substituting formula for $p^{(\kappa)}$. 
\end{proof}
\begin{proof}[Proof of Proposition \ref{prop:max}]
Applying induced-norm to the alternative expression of $\hat A$ in \eqref{hatSmatrix} yields 
\begin{align*}
&\|\hat S\|_2=\|AWB\|_2\leq \|A\|_2\|B\|_2\|W\|_2,
\end{align*}
where $W\dot =DC^2D^T$. $W$ returns a $n\times n$ diagonal matrix with nonnegative entries $(\Theta^{(\kappa)}_{r_i,d_i},\Theta^{(\kappa)}_{r_i,d_i})$ for $1\leq i\leq c$ and $1\leq d_i\leq |\Theta^{(\kappa)}_{r_i}|$ taking value on a ratio of a scalar, which indicates how many times $\Theta^{(\kappa)}_{r_i}$ has been drawn, to $\frac{c}{k}$. The probability that at least one of the indices has been drawn $s^{(k)}_c$ times must be smaller than
$$\sum_{\ell=1}^k \mathcal{C}_{c}^{s^{(k)}_c} (\frac{1}{k})^{s^{(k)}_c}(1-\frac{1}{k})^{c-s^{(k)}_c},$$
where $\mathcal{C}_c^{s^{(k)}_c}$ is the  number of combinations of choosing $s^{(k)}_c$ among $c$ In summary, we have that
\begin{align*}
&\mathbb{P}_p(\text{no indices has been drawn more than }s^{(k)}_c-1\ \text{times})\\
&\geq  1-\sum_{i\geq s^{(k)}_c}\mathbb{P}_p(\text{at least one index has been drawn more than } i\ \text{times})\\
&\geq 1-\sum_{i\geq s^{(k)}_c}\sum_{\ell=1}^k \mathcal{C}_{c}^{i} \big(\frac{1}{k}\big)^{i}\big(1-\frac{1}{k}\big)^{c-i}\\
&=  1-\sum_{i-1\geq s^{(k)}_c-1} \frac{c}{i}\mathcal{C}_{c-1}^{i-1} \big(\frac{1}{k}\big)^{i-1}(1-\frac{1}{k})^{(c-1)-(i-1)}\\
& \geq 1-\frac{c}{s^{(k)}_c}\big(1-F_{\text{b}}\big(s^{(k)}_c-2;c-1,\frac{1}{k}\big)\big)\geq 0.99.
\end{align*}
\end{proof}
\begin{proof}[Proof of Corollary \ref{cor:argvar}]
First note that $p^{(o)}$ is given in \eqref{optimalmatrix} as
 \begin{align}\label{eqn:optimaln}
p^{(o)}_\ell = \frac{\|A_{(,\ell)}\|_F\|B_{(\ell,)}\|_F}{\sum_{\ell=1}^k\|A_{(,\ell)}\|_F\|B_{(\ell,)}\|_F}= \frac{\|A_{(,\ell)}\|_2\|B_{(\ell,)}\|_2}{\sum_{\ell=1}^k\|A_{(,\ell)}\|_2\|B_{(\ell,)}\|_2}.
\end{align}
Now fix $\Theta^{(\text{pair})}$. Then substituting the expression of $p^{(o)}$ into \eqref{eqn:var} yields
\begin{align*}
&\mathbb{E}_{p^{(o)}}[\|AB-\hat S\|_F]=  \frac{1}{c}\sum_{\ell=1}^{\frac{n}{2}}\frac{1}{p^{(o)}_{I_{2\ell-1}}+p^{(o)}_{I_{2\ell}}}\big\|A_{(,\{I_{2\ell-1},I_{2\ell}\})}B_{(\{I_{2\ell-1},I_{2\ell}\},)}\big\|_F^2-\frac{\|AB\|_F^2}{c}\\
\leq&\frac{1}{c}\sum_{\ell=1}^{\frac{n}{2}}\frac{1}{p^{(o)}_{I_{2\ell-1}}+p^{(o)}_{I_{2\ell}}}\big(\|A_{(,I_{2\ell-1})}\|_2\|B_{(I_{2\ell-1},)}\|_2+\|A_{(,I_{2\ell})}\|_2\|B_{(I_{2\ell},)}\|_2)^2-\frac{\|AB\|_F^2}{c}\\
=& \frac{1}{c}\Big(\sum_{j=1}^n\|A_{(,I_j)}\|_2\|B_{(I_j,)}\|_2\Big)\sum_{\ell=1}^{\frac{n}{2}}\frac{\big(\|A_{(,I_{2\ell-1})}\|_2\|B_{(I_{2\ell-1},)}\|_2+\|A_{(,I_{2\ell})}\|_2\|B_{(I_{2\ell},)}\|_2)^2}{\|A_{(,I_{2\ell-1})}\|_2\|B_{(I_{2\ell-1},)}\|_2+\|A_{(,I_{2\ell})}\|_2\|B_{(I_{2\ell},)}\|_2}-\frac{\|AB\|_F^2}{c}\\
=&V_{p^{(o)}}(AB).
\end{align*}
\end{proof}
\begin{proof}[Proof of Corollary \ref{cor:lowerbound}] The argument is similar as in the proof of Proposition \ref{prop:2norm}. Denote by 
Define a random index $r\in\{1,2,\ldots,n\}$, with $\mathbb{P}_{p^{(o)}}(r=\ell)=p^{(o)}_\ell$ and thus a random rectangular matrix $D=\frac{1}{p_r}A_{(,r)}B_{(r,)}$. Meanwhile, define another random index $t\in\{1,2,\ldots,\frac n 2\}$, with $\mathbb{P}_{p^{(o)}}(t=i)=p^{(o)}_{I_i}+p^{(o)}_{I_{n-i}}$ and thus a random  rectangular matrix $G=\frac{1}{(p^{(o)}_{I_{t}}+p^{(o)}_{I_{n-t}})}A_{(,\{I_{t},I_{n-t}\})}B_{(\{I_t,I_{n-t}\},)}$. It is easy to check that $D-AB$ and $G-AB$ are zero mean matrices. Now assume we have $c$ copies of $D$ and $G$ termed as $D_1, D_2,\ldots,D_c$ and $G_1, G_2,\ldots,G_c$. To apply with noncommutative Bernstein inequality and eventually justify $\big\|\frac{1}{c}\sum_{k=1}^c G_k-AB\big\|_2\leq \epsilon$ has larger lower bound than $\big\|\frac{1}{c}\sum_{k=1}^c D_k-AB\big\|_2\leq \epsilon$, we have to show two things:
\begin{enumerate}
\item if $\|D_k-AB\|_2\leq M_1$ and $\|G_k-AB\|_2\leq M_2$ a.s, for all $k$, then $M_2\leq M_1$;
\item  if define $$u_1=\max \big\{\|\mathbb{E}_{p^{(o)}}[(D-AB)^T(D-AB)]\|_2, \|\mathbb{E}_{p^{(o)}}[(D-AB)(D-AB)^T]\|_2\big\}$$ and 
$$u_2=\max \big\{\|\mathbb{E}_{p^{(o)}}[(G-AB)^T(G-AB)]\|_2, \|\mathbb{E}_{p^{(o)}}[(G-AB)(G-AB)^T]\|_2\big\},$$ then $u_2\leq u_1$.
\end{enumerate}
In the following we are going to verify the two points:
\begin{enumerate}
\item first note that
\begin{align*}
&\|AB-D\|_2=\big\|AB-\frac{1}{p^{(o)}_r}A_{(,r)}B_{(r,)}\big\|_2\leq \big\|\sum_{\ell=1,\ell\neq r}^n A_{(,\ell)}B_{(\ell,)}\big\|_2+\frac{(1-p^{(o)}_r)}{p^{(o)}_r}\big\|A_{(,r)}B_{(r,)}\big\|_2\\
&\quad \leq \sum_{\ell=1,\ell\neq r}^n\|A_{(,\ell)}\|_2\|B_{(\ell,)}\|_2+\frac{(1-p^{(o)}_r)}{p^{(o)}_r}\|A_{(,r)}\|_2\|B_{(r,)}\|_2=2\sum_{\ell=1,\ell\neq r}^n\|A_{(,\ell)}\|_2\|B_{(\ell,)}\|_2,
\end{align*}
where the last line holds by plugging $p^{(o)}$, i.e., Eqn. \eqref{optimalmatrix1}, into the penultimate line. Now define 
$$M_1:=2\max_{r}\sum_{\ell=1,\ell\neq r}^n\|A_{(,\ell)}\|_2\|B_{(\ell,)}\|_2.$$
Besides, we have that (assume $r$ is odd)
\begin{align*}
&\|AB-G\|_2=\Big\|AB-\frac{1}{(p^{(o)}_{I_r}+p^{(o)}_{I_{r+1}})}A_{(,\{I_r,I_{r+1}\})}B_{(\{I_r,I_{r+1}\},)}\Big\|_2\\
& \leq \Big\|\sum_{\ell=1,\text{odd},\ell\neq r}^{n} A_{(,\{I_\ell,I_{\ell+1}\})}B_{(\{I_\ell,I_{\ell+1}\},)}\Big\|_2+\frac{(1-p^{(o)}_{I_r}-p^{(o)}_{I_{r+1}})}{(p^{(o)}_{I_r}+p^{(o)}_{I_{r+1}})}\big\|A_{(,\{I_r,I_{r+1}\})}B_{(\{I_r,I_{r+1}\},)}\big\|_2\\
& \leq \sum_{\ell=1,\text{odd},\ell\neq r}^{n } (\|A_{(,I_{\ell})}\|_2\|B_{(I_{\ell},)}\|_2+\|A_{(,I_{\ell+1})}\|_2\|B_{(I_{\ell+1},)}\|_2)\\
&\quad +\frac{(1-p^{(o)}_{I_r}-p^{(o)}_{I_{r+1}})}{(p^{(o)}_{I_r}+p^{(o)}_{I_{r+1}})}(\|A_{(,I_r)}\|_2\|B_{(I_r,)}\|_2+\|A_{(,I_{r+1})}\|_2\|B_{(I_{r+1},)}\|_2)\\
& =2\sum_{\ell=1,\ell\neq I_r, I_{r+1}}^n\|A_{(,\ell)}\|_2\|B_{(\ell,)}\|_2,
\end{align*}
where the last line holds by plugging $p^{(o)}$, i.e., Eqn. \eqref{eqn:optimaln}, into the penultimate line. Now define 
$$M_2:=2\max_r \sum_{\ell=1,\ell\neq I_r,I_{r+1}}^n\|A_{(,\ell)}\|_2\|B_{(\ell,)}\|_2.$$
Apparently we can conclude that $M_2\leq M_1$. Note that, among all possible pairing strategies, $M_2$ achieves its minimum only if pairing like enhanced pairwise partition. 
\item  let us first examine $\|\mathbb{E}_{p^{(o)}}[(D-AB)^T(D-AB)]\|_2$, the bound of which can be applied to $\|\mathbb{E}_{p^{(o)}}[(D-AB)(D-AB)^T]\|_2$ as well. The technique is quite similar as that in the proof of proposition \ref{prop:2norm} and the estimate from last step is adopted as well.
\begin{align*}
&\|\mathbb{E}_{p^{(o)}}[(D-AB)^T(D-AB)]\|_2\\
&=\Big\| \sum_{\ell=1}^n \Big\{p^{(o)}_\ell \big[\big(\frac{1}{p^{(o)}_\ell}A_{(,\ell)}B_{(\ell,)}-AB\big)^T\big(\frac{1}{p^{(o)}_\ell}A_{(,\ell)}B_{(\ell,)}-AB\big)\big]\Big\}\Big\|_2\\
&\leq  \sum_{\ell=1}^n \Big\{p^{(o)}_\ell \big\|\frac{1}{p^{(o)}_\ell}A_{(,\ell)}B_{(\ell,)}-AB\big\|^2_2\Big\}\\
&\leq 4\sum_{\ell=1}^n \Big\{p^{(o)}_\ell \Big(\sum_{i=1,i\neq \ell}^n\|A_{(,i)}\|_2\|B_{(i,)}\|_2\Big)^2\Big\}\\
&=\frac{4}{\sum_{j=1}^n\|A_{(,j)}\|_2\|B_{(j,)}\|_2}\sum_{\ell=1}^n\Big\{ \|A_{(,\ell)}\|_2\|B_{(\ell,)}\|_2\Big(\sum_{i=1,i\neq \ell}^n\|A_{(,i)}\|_2\|B_{(i,)}\|_2\Big)^2\Big\},
\end{align*}
where the last inequality is true by substituting bound from last step, and last equality holds by plugging $p^{(o)}$, i.e., Eqn. \eqref{optimalmatrix1}, into the penultimate line. Now define 
\begin{align*}
&u_1:=\frac{4}{\sum_{j=1}^n\|A_{(,j)}\|_2\|B_{(j,)}\|_2}\sum_{\ell=1}^n \Big\{\|A_{(,\ell)}\|_2\|B_{(\ell,)}\|_2\Big(\sum_{i=1,i\neq \ell}^n\|A_{(,i)}\|_2\|B_{(i,)}\|_2\Big)^2\Big\}.
\end{align*}
Regarding the estimation for $\|\mathbb{E}_{p^{(o)}}[(G-AB)^T(G-AB)]\|_2$, similar argument applies and we can deduce that
\begin{align*}
&u_2:=\frac{4}{\sum_{j=1}^n\|A_{(,j)}\|_2\|B_{(j,)}\|_2} \sum_{\ell=1}^{\frac n 2} \Big\{\big(\|A_{(,I_{2\ell-1})}\|_2\|B_{(I_{2\ell-1},)}\|_2\\
&\quad\quad +\|A_{(,I_{2\ell})}\|_2\|B_{(I_{2\ell},)}\|_2\big)\Big(\sum_{i=1,i\neq I_{2\ell-1},I_{2\ell}}^n\|A_{(,i)}\|_2\|B_{(i,)}\|_2\Big)^2\Big\}.
\end{align*}
It is not hard to conclude that $u_1\geq u_2$. Besides, among all pairing strategies, $u_2$ its minimum only if pair partition is given as enhanced pairwise partition. 
\end{enumerate}
\end{proof}

\end{document}